\documentclass[11pt]{article}
\usepackage{mylatexsetting}
\usepackage{longtable}
\def\mytitle{Kernel Characterisations of Stochastic Orders Within Parametric Density Families}

\def\myabstract{
We develop kernel criteria for the likelihood-ratio, hazard-rate, usual stochastic, and relative log-concavity orders in parametric families of univariate probability laws with densities. The score is the derivative of the log density with respect to the parameter, and a kernel equals the score up to an additive term depending only on the parameter. Kernel monotonicity gives likelihood-ratio order, kernel concavity gives relative log-concavity, and two tail-conditional mean inequalities give the hazard-rate and usual stochastic orders. The same construction applies along joint-parameter paths and to comparisons between two laws whose densities admit parameter-dependent factors, where the log-factor ratio is used as the kernel. For compound sums with a random number of i.i.d. terms, the induced kernel is the posterior mean of the kernel of the summand count. The applications recover standard one-parameter orderings, give likelihood-ratio comparisons for compound laws, and handle nonmonotone examples through the tail-conditional criteria.
}

\title{\mytitle}
\author{Zakaria Derbazi\footnote{Correspondence address: z.derbazi@qmul.ac.uk}\\{\itshape\small Queen Mary University of London}}
\date{\today}

\begin{document}

\maketitle

\begin{abstract}
\myabstract
\end{abstract}



\bigskip
\section{Introduction}
\label{sec:introduction}

This paper develops kernel criteria for the likelihood-ratio, relative log-concavity, hazard-rate, and usual stochastic orders in parametric families of univariate probability measures. For densities differentiable in the parameter, the score is the derivative of the log density with respect to that parameter. A kernel is a function on the support whose centred version under the current law equals the score. 

The kernel yields characterisations of all four stochastic orders between laws in the family, and is often simpler to analyse than the likelihood ratio. Monotonicity gives the likelihood-ratio order, concavity gives relative log-concavity, and two tail-conditional mean inequalities give the hazard-rate and the usual stochastic orders. The construction also applies to direct comparisons of two laws whose densities admit parameter-dependent factors. The difference of the two log factors then replaces the derivative of the log density with respect to the parameter in the same tests.

The kernel lives at the level of parameter-dependent factors, before centring produces the score. This makes the same criteria portable to joint-parameter paths, pairwise factor-form comparisons, and compound laws.

Our main contributions are as follows:
\begin{enumerate}[label=\textup{(\roman*)}]
\item Kernel criteria for comparing two members of the same parametric family in the likelihood-ratio, relative log-concavity, hazard-rate, and usual stochastic orders, via a uniform four-order characterisation in terms of the kernel and its tail-conditional means (Proposition~\ref{prop:weighted-identity}, Lemma~\ref{lem:three-levels}, Theorem~\ref{thm:four-criteria}).

\item A superlevel-set criterion that establishes hazard-rate and usual stochastic ordering when the kernel is nonmonotone or nonconcave (Proposition~\ref{prop:kernel-superlevel}, Corollary~\ref{cor:concave-kernel}, Corollary~\ref{cor:unimodal-kernel}).

\item Extensions to joint-parameter and interpolation paths, and to comparisons of laws not belonging to a common parametric family, where the relevant kernel is either the path derivative or the difference of two log factors.

\item A compound-construction extension by which the compound law inherits a kernel from the underlying counting law through a posterior-averaging identity, extending the compound-geometric comparison of~\cite{XiaLv2024} beyond relative log-concavity (Proposition~\ref{prop:compound-score}, Corollary~\ref{cor:compound-kernel}, Lemma~\ref{lem:cond-mean-monotone}, Proposition~\ref{prop:compound-monotone}).

\item A catalogue of examples illustrating the criteria on classical discrete and continuous distributions (Tables~\ref{tab:common-kernels} and~\ref{tab:compound-slopes}).
\end{enumerate}

Our approach sits alongside a substantial literature on criteria formulated directly in terms of the likelihood ratio. For instance, stochastic ordering for exponential families is treated in~\cite{Yu2009}, where the hazard-rate and usual stochastic orders are reduced to endpoint conditions under relative log-concavity of the likelihood ratio. The paper~\cite{KlenkeMattner2010} catalogues comparison methods across several classical discrete laws. Comparisons involving compound discrete laws are studied in~\cite{XiaLv2024} through relative log-concavity, and threshold-type results for parameter-mixed laws appear in~\cite{MisraSinghHarner2003,AlamatsazAbbasi2008,PudprommaratBodhisuwan2012}. Shape and sign-pattern conditions on the likelihood ratio are used in~\cite{DerbaziPairwise} to obtain endpoint criteria for the hazard-rate and usual stochastic orders, together with likelihood-ratio endpoint tests under relative log-concavity.

The paper is organised as follows. Section~\ref{sec:prelim} fixes notation and conventions and recalls the four stochastic orders. Section~\ref{sec:kernel} develops the criteria for comparing two laws in the same parametric family, together with the superlevel-set criterion and the compound-law construction. Section~\ref{sec:applications} applies the criteria to classical discrete and continuous distributions, organised by how the kernel is obtained: from the score of a parametric density family, from the derivative along a path in a multi-parameter family, from a direct log-factor difference between two laws, or from a posterior average for compound laws. Section~\ref{sec:discussion} concludes with a short discussion on two further directions.


\section{Notation and Preliminaries}
\label{sec:prelim}

\subsection{Conventions}\label{subsec:convention}
Throughout, the probability measures under consideration are dominated by a $\sigma$-finite measure $\mu$ on a totally ordered space $E$, and densities are understood as Radon--Nikodym derivatives with respect to $\mu$. The standard choices are counting measure on an integer interval and Lebesgue measure on an interval of $\R$.  We write \(\delta_x\) for the unit point mass at \(x\). In examples with both an atom at zero and an absolutely continuous component, we use the dominating measure \(\delta_0+\dx\). For a probability measure $P$ on $E$ we write $f_P\coloneqq \dx[]P/\dx[\mu]$ for its density, $\bar F_P(x)\coloneqq P([x,\infty)\cap E)$ for its survival function, \(h_P(x)\coloneqq f_P(x)/\bar F_P(x)\) for its hazard rate (where \(\bar F_P(x)>0\)), and $\supp(P)\coloneqq \{f_P>0\}$ up to $\mu$-null sets.

For a pair of probability measures \(P,Q\), set \(\rho\coloneqq P+Q\) and \(S\coloneqq \supp(\rho)\). Their likelihood ratio is
\begin{equation}
\label{def.ell}
\ell(x)\coloneqq  \frac{(\dx[]P/\dx[]\rho)(x)}{(\dx[]Q/\dx[]\rho)(x)}, \qquad x\in S,
\end{equation}
with the conventions \(a/0\coloneqq +\infty\) for \(a>0\) and \(0/0\coloneqq 0\). Where both \(\mu\)-densities are positive, this agrees with the ordinary density ratio \(f_P/f_Q\).

A parametric family of probability measures is written \((P_\vartheta)_{\vartheta\in\mathcal V}\), where \(\mathcal V\subseteq\R^d\). It is specified by a corresponding family of \(\mu\)-densities \((f_\vartheta)_{\vartheta\in\mathcal V}\), so that for measurable  $A\subseteq E$
\[
P_\vartheta(A)=\int_A f_\vartheta(x)\,\mu(\dx[x]).
\]
We reserve \emph{parametric family} for the indexed probability measures being compared, and \emph{density family} for their densities with respect to \(\mu\). Terms such as exponential family, generalised power-series family (abbreviated GPS), and generalised hypergeometric family are used in their usual sense, as classes of parametric families with a common density or mass-function structure.

The criteria below are one-dimensional. Thus, when \(d>1\), all parameters except the one under consideration are held fixed, and the resulting parametric density family is written \((f_\nu)_{\nu\in\mathcal I}\), with associated probability measures \((P_\nu)_{\nu\in\mathcal I}\) and \(\mathcal I\subseteq\R\) an interval. If several parameters vary together, we write the resulting one-parameter path as \((P_t)_{t\in\mathcal T}\), with \(\mathcal T\subseteq\R\) an interval. The common support of the laws under comparison is denoted \(J\subseteq E\). We use $X,Y$ for random variables with values in $E$, distributed according to the law under consideration, and take expectations under that law.

We write $\N_0\coloneqq \{0\}\cup\N$. In the discrete case, sets of the form $[a,b]\coloneqq \{a,\dots,b\}$ and $[a,\infty)\coloneqq \{a,a{+}1,\dots\}$ are intervals in $\N_0$. Forward differences are $\Delta h(k)\coloneqq h(k+1)-h(k)$ and $\Delta^2 h(k)\coloneqq h(k+2)-2h(k+1)+h(k)$. We call $k$ admissible for $\Delta^2$ when $k,k+1,k+2$ all lie in the support. We use the Pochhammer symbol $(a)_k\coloneqq a(a+1)\cdots(a+k-1)$ with $(a)_0\coloneqq 1$. Recall that $\partial_a\log(a)_k=\psi(a+k)-\psi(a)$, where $\psi(x)\coloneqq \Gamma'(x)/\Gamma(x)$  is the digamma function. 
The indicator of a set $A\subseteq E$ is $\one\{A\}$.

\begin{assumption}[Differentiation under the integral]
\label{asm:diff-under-integral}
For every parametric density family \((f_\nu)_{\nu\in\mathcal I}\) considered below, with associated probability measures \((P_\nu)_{\nu\in\mathcal I}\) and \(\mathcal I\subseteq\R\) an interval, the density \(f_\nu(x)\) is positive on a common interval \(J\subseteq E\) and \(\nu \mapsto\log f_\nu(x)\) is \(C^1\) on \(\mathcal I\) for every \(x\in J\). For every compact interval \(B\subseteq\mathcal I\), there exists a majorant \(m\in L^1(J,\mu)\) such that
\[
\sup_{\nu\in B}\bigl|\partial_\nu f_\nu(x)\bigr|\le m(x)
\qquad \mu\text{-a.e.}
\]
\end{assumption}
This justifies differentiation under the integral (or sum) sign, which we use without further mention. In particular, once the score is defined in Section~\ref{sec:kernel}, the maps \(\nu\mapsto s_\nu(x)\) and \(\nu\mapsto \partial_\nu\log\bar F_\nu(x)\) are continuous at every fixed \(x\in J\) with \(\bar F_\nu(x)>0\). This continuity is deployed in the necessity direction of Theorem~\ref{thm:four-criteria}.

\subsection{Shape notions and stochastic orders}\label{subsec:orders}
\begin{defn}[Log-concavity, PF$_2$, and TP$_2$]
A positive function \(a:J\to(0,\infty)\) on an interval \(J\subseteq E\) is \emph{log-concave} if \(\log a\) is concave on \(J\). In the discrete case this is equivalent to \(\Delta^2\log a(k)\le0\) for every admissible \(k\in J\) and in the continuous case it is equivalent to \((\log a)''(x)\le0\) wherever the second derivative exists.

For a sequence \(b=(b_k)_{k\in\Z}\), we say that \(b\) is a \emph{P\'olya frequency sequence of order 2} (PF$_2$) if \(b_k\ge0\) for all \(k\), the support \(\{k:b_k>0\}\) is an interval in \(\Z\), and \(b\) is log-concave on its support.  A probability distribution \(F\) on \(\N_0\) is PF$_2$ if its mass sequence, extended by zero outside its support, is PF$_2$.

A nonnegative kernel \(M:I\times J\to[0,\infty)\) on two totally ordered sets is \emph{totally positive of order 2} (TP$_2$) if
\[
M(i,j)M(i',j')\ge M(i,j')M(i',j)
\]
whenever \(i\le i'\) and \(j\le j'\). See~\cite{Karlin1968} for the general total positivity theory.
\end{defn}

We recall the four orders used throughout. See~\cite{ShakedShanthikumar,Whitt1985} for further background.

\begin{defn}\label{def:orders}
Let \(P\) and \(Q\) be probability measures on \(E\), with likelihood ratio \(\ell\) as in~\eqref{def.ell}.
\begin{enumerate}[label=\textup{(\roman*)}]

\item \emph{Usual stochastic order} ($\st$): \(P\st Q\) if
\(\bar F_P(x)\le\bar F_Q(x)\) for every \(x\in E\).

\item \emph{Hazard-rate order} ($\hr$): \(P\hr Q\) if \(\bar F_P(x)/\bar F_Q(x)\) is nonincreasing wherever \(\bar F_Q(x)>0\) (equivalently, \(h_P(x)\ge h_Q(x)\) wherever both hazard rates are defined).

\item \emph{Likelihood-ratio order} ($\lr$): \(P\lr Q\) if \(\ell\) is nonincreasing on \(\supp(P)\cup\supp(Q)\).

\item \emph{Relative log-concavity} ($\lc$): when \(\supp(P)\) is an interval contained in \(\supp(Q)\), \(P\lc Q\) if \(\log\ell\) is concave on \(\supp(P)\).
\end{enumerate}
\end{defn}

In Section~\ref{sec:kernel}, the criteria are stated in the order \(\lr,\lc,\st,\hr\), matching the corresponding kernel conditions: monotonicity, concavity, and two tail-conditional inequalities.


\section{Kernel Criteria and Compound Laws}
\label{sec:kernel}

Throughout this section, let \((f_\nu)_{\nu\in\mathcal I}\) be a parametric density family satisfying Assumption~\ref{asm:diff-under-integral}, with associated probability measures \((P_\nu)_{\nu\in\mathcal I}\) and common support interval \(J\subseteq E\). The \emph{score} is
\begin{equation}\label{def.score}
s_\nu(x)\coloneqq \partial_\nu\log f_\nu(x),\qquad x\in J,
\end{equation}
and is centred under $P_\nu$:
\[
\int_J s_\nu\,\dx[]P_\nu=0.
\]
A \emph{kernel} (with respect to $P_\nu$) is any measurable, \(P_\nu\)-integrable function \(K_\nu:J\to\R\) satisfying
\begin{equation}\label{def.log.companion}
s_\nu(x)=K_\nu(x)-\int_J K_\nu(y)\,\dx[]P_\nu(y),\qquad x\in J.
\end{equation}
Since the centring term \(\int_J K_\nu\,\dx[]P_\nu\) is a constant independent of $x$, \(K_\nu\) and \(s_\nu\) have the same properties on \(J\) whenever those properties are invariant under additive constants. The score itself is the unique centred kernel. Kernels are determined only up to addition of a function of the parameter. We state the ordering criteria in kernel form and pass to the score whenever a density-level proof is more convenient.
\begin{rem}
Suppose the density can be expressed as $f_\nu(x)={w_\nu(x)}/{A(\nu)}$, where $A(\nu)\coloneqq \int_J w_\nu(y)\,\mu(\dx[y])$. Then
	\begin{equation}\label{def.kernel}
		K_\nu(x)\coloneqq \partial_\nu\log w_\nu(x),\qquad x\in J,
	\end{equation}
	is a kernel. Indeed, differentiating \(A(\nu)=\int_J w_\nu\,\dx[\mu]\) under the integral gives \(\partial_\nu\log A(\nu)=\int_J K_\nu\,\dx[]P_\nu\), hence
	\[
	s_\nu(x) =\partial_\nu\log w_\nu(x)-\partial_\nu\log A(\nu) =K_\nu(x)-\int_J K_\nu\,\dx[]P_\nu.
	\]
\end{rem}

\subsection{Density, survival, and hazard identities}\label{subsec:density-survival-hazard}

\begin{defn}[Regularity condition]
\label{def:regularity-condition}
A measurable function \(u:J\to[0,\infty)\) satisfies the \emph{regularity condition} if, for every compact interval \(B\subseteq\mathcal I\), the integral \(\int_J u\,\dx[]P_\nu\) is finite and positive for every \(\nu\in B\), and there exists a majorant \(m\in L^1(J,\mu)\) such that
\[
\sup_{\nu\in B}\bigl|u(x)\,\partial_\nu f_\nu(x)\bigr|\le m(x)
\qquad \mu\text{-a.e. } x\in J.
\]
\end{defn}

By Assumption~\ref{asm:diff-under-integral}, every bounded measurable \(u\) satisfies the majorant condition of Definition~\ref{def:regularity-condition}. Thus, tail indicators \(u_y(x)=\one\{x\ge y\}\) satisfy the regularity condition whenever \(\bar F_\nu(y)>0\) for every \(\nu\) in the compact parameter interval under consideration.
\begin{prop}
\label{prop:weighted-identity}
Suppose \(u:J\to[0,\infty)\) satisfies the regularity condition, and write \(P_\nu^u\) for the probability measure defined by
\[
\frac{\dx[]P_\nu^u}{\dx[]P_\nu}(x)=\frac{u(x)}{\int_J u\,\dx[] P_\nu },\qquad x\in J.
\]
Then
\begin{equation}\label{eq:weighted-mean}
\partial_\nu\log\int_J u\,\dx[]P_\nu
=\int_J s_\nu\,\dx[]P_\nu^u
=\int_J K_\nu\,\dx[]P_\nu^u-\int_J K_\nu\,\dx[]P_\nu.
\end{equation}
If \(v:J\to[0,\infty)\) also satisfies the same regularity condition, and \(P_\nu^v\) is defined analogously, then
\[
\partial_\nu\log\dfrac{\displaystyle\int_J u\,\dx[]P_\nu}{\displaystyle\int_J v\,\dx[]P_\nu}
=\int_J s_\nu\,\dx[]P_\nu^u-\int_J s_\nu\,\dx[]P_\nu^v
=\int_J K_\nu\,\dx[]P_\nu^u-\int_J K_\nu\,\dx[]P_\nu^v.
\]
\end{prop}

\begin{proof}
Since the regularity condition on \(u\) justifies differentiating under the integral sign, it follows that 
\[
\partial_\nu\int_J u\,\dx[]P_\nu
=\int_J u(x)\,\partial_\nu f_\nu(x)\,\mu(\dx)
=\int_J u(x)\,s_\nu(x)\,\dx[]P_\nu.
\]
Dividing by \(\int_J u\,\dx[]P_\nu\) yields
\[
\partial_\nu\log\int_J u\,\dx[]P_\nu
=\frac{\int_J u(x)\,s_\nu(x)\,\dx[]P_\nu}
      {\int_J u\,\dx[]P_\nu}
=\int_J s_\nu\,\dx[]P_\nu^u.
\]
Substituting~\eqref{def.log.companion} gives the kernel form. The general case follows by applying~\eqref{eq:weighted-mean} to \(u\) and
\(v\) and taking the difference.
\end{proof}

\begin{lem}
\label{lem:three-levels}
Let $X$ be a random variable with law $P_\nu$. For every \(x\in J\) with \(\bar F_\nu(x)>0\),
\begin{alignat}{2}
\partial_\nu\log f_\nu(x) &=s_\nu(x)=K_\nu(x)-\E[K_\nu(X)], \label{eq:lr-kernel}\\
\partial_\nu\log\bar F_\nu(x) &=\E[s_\nu(X)\mid X\ge x]=\E[K_\nu(X)\mid X\ge x]-\E[K_\nu(X)], \label{eq:st-kernel}\\
\partial_\nu\log h_\nu(x) &=s_\nu(x)-\E[s_\nu(X)\mid X\ge x]=K_\nu(x)-\E[K_\nu(X)\mid X\ge x]. \label{eq:hr-kernel}
\end{alignat}
\end{lem}

\begin{proof}
The density-level identity corresponds to definitions \eqref{def.score} and \eqref{def.log.companion} expressed through $X$. For the survival identity, take \(u(y)=\one\{y\ge x\}\) in Proposition~\ref{prop:weighted-identity}. Then $\int_J u\,\dx[]P_\nu=\bar F_\nu(x),$
and the corresponding new law is the conditional law of \(X\) given \(X\ge x\). Hence $\partial_\nu\log\bar F_\nu(x) =\E[s_\nu(X)\mid X\ge x].$ Substituting \(s_\nu\) with its expression in \eqref{eq:lr-kernel} gives the right-hand side of \eqref{eq:st-kernel}. 
Finally, subtract~\eqref{eq:st-kernel} from~\eqref{eq:lr-kernel} to get \eqref{eq:hr-kernel}. This completes the proof.
\end{proof}

Lemma~\ref{lem:three-levels} expresses each of the density, survival, and hazard derivatives as a difference between two of the three quantities \(K_\nu(x)\), \(\E[K_\nu(X)\mid X\ge x]\), and \(\E[K_\nu(X)]\). Integrating these local identities over a parameter interval gives the pairwise comparison functions
\[
L_{\nu_2,\nu_1}(x) \coloneqq \log\frac{f_{\nu_2}(x)}{f_{\nu_1}(x)},\quad T_{\nu_2,\nu_1}(x) \coloneqq \log\frac{\bar F_{\nu_2}(x)}{\bar F_{\nu_1}(x)}, \quad H_{\nu_2,\nu_1}(x) \coloneqq \log\frac{h_{\nu_2}(x)}{h_{\nu_1}(x)},
\]
defined wherever the corresponding denominators are positive. In particular, for $\nu_1<\nu_2$ in \(\mathcal I\) and $x\in J$, the log-likelihood ratio is given by
\begin{equation}\label{eq:pair-density-int}
L_{\nu_2,\nu_1}(x) =\int_{\nu_1}^{\nu_2}s_\nu(x)\,\dx[\nu] =\int_{\nu_1}^{\nu_2} \bigl(K_\nu(x)-\E[K_\nu(X)]\bigr)\,\dx[\nu],
\end{equation}
with analogous representations of $T_{\nu_2,\nu_1}$ and $H_{\nu_2,\nu_1}$ via the kernel and its tail-conditional means. These representations lead to the ordering criteria of the next subsection.

\subsection{Comparing laws in parametric families}\label{subsec:one-parameter-criteria}

\begin{thm}
\label{thm:four-criteria}
\leavevmode
\begin{enumerate}[label=\textup{(\roman*)}]
\item \(P_{\nu_1}\lr P_{\nu_2}\) for all \(\nu_1\le\nu_2\) in \(\mathcal I\) if and only if \(K_\nu\) is nondecreasing on \(J\) for every \(\nu\in\mathcal I\).

\item \(P_{\nu_2}\lc P_{\nu_1}\) for all \(\nu_1\le\nu_2\) in \(\mathcal I\) if and only if \(K_\nu\) is concave on \(J\) for every \(\nu\in\mathcal I\).

\item \(P_{\nu_1}\st P_{\nu_2}\) for all \(\nu_1\le\nu_2\) in \(\mathcal I\) if and only if \(\E[K_\nu(X)\mid X\ge x]\ge\E[K_\nu(X)]\) for every \(\nu\in\mathcal I\) and \(x\in J\) with \(\bar F_\nu(x)>0\).

\item \(P_{\nu_1}\hr P_{\nu_2}\) for all \(\nu_1\le\nu_2\) in \(\mathcal I\) if and only if \(K_\nu(x)\le\E[K_\nu(X)\mid X\ge x]\) for every \(\nu\in\mathcal I\) and \(x\in J\) with \(\bar F_\nu(x)>0\).
\end{enumerate}
\end{thm}
\begin{proof}
Sufficiency of \textup{(i)} and \textup{(ii)} follows from~\eqref{eq:pair-density-int}: the centring term is irrelevant, and an integral over \(\nu\) of nondecreasing or concave functions of \(x\) is again nondecreasing or concave. For \textup{(iii)} and \textup{(iv)}, integrate~\eqref{eq:st-kernel} or~\eqref{eq:hr-kernel} over $[\nu_1,\nu_2]$ and read off the sign.

For necessity,  identity~\eqref{def.log.companion} gives
\begin{equation}
\label{eq:kernel-score-shape}
K_\nu(x_1)-K_\nu(x_2)=s_\nu(x_1)-s_\nu(x_2),\qquad x_1,x_2\in J,
\end{equation}
and analogously for second differences and tail-conditional comparisons:
\begin{equation}
\label{eq:kernel-score-tail}
\E[K_\nu(X)\mid X\ge x]-\E[K_\nu(X)] =\E[s_\nu(X)\mid X\ge x]
\end{equation}
on \(\{x:\bar F_\nu(x)>0\}\), where the right-hand side equals \(\partial_\nu\log\bar F_\nu(x)\) by Lemma~\ref{lem:three-levels}. Both right-hand sides of (\ref{eq:kernel-score-shape}) and (\ref{eq:kernel-score-tail}) are continuous in \(\nu\) at every fixed \(x\) by Assumption~\ref{asm:diff-under-integral}. The kernel expressions on the left are therefore also continuous in \(\nu\), independently of the choice of $K_\nu$. The rest of the proof proceeds by contradiction, as follows:

\emph{Assertion~\textup{(i)}.} Suppose \(P_{\nu_1}\lr P_{\nu_2}\) for all \(\nu_1\le\nu_2\), and assume by contradiction that for some \(\nu_0 \in \mathcal I\), \(K_{\nu_0}\) is not nondecreasing on \(J\). By~\eqref{eq:kernel-score-shape}, this means $s_{\nu_0}$ is not nondecreasing, so there exist \(x_1<x_2\) in \(J\) such that \(s_{\nu_0}(x_1)>s_{\nu_0}(x_2)\). By continuity of $\nu\mapsto s_\nu(x_i)$, there is an \(\varepsilon>0\) and an interval $I_\varepsilon\subseteq\mathcal I$ containing $\nu_0$ such that $s_\nu(x_1)-s_\nu(x_2)\ge\varepsilon$ for all $\nu\in I_\varepsilon$. Choosing \(\nu_1<\nu_2\) in \(I_\varepsilon\) gives $L_{\nu_2,\nu_1}(x_1)-L_{\nu_2,\nu_1}(x_2)
=\int_{\nu_1}^{\nu_2}\bigl(s_\nu(x_1)-s_\nu(x_2)\bigr)\,\dx[\nu]
\ge \varepsilon(\nu_2-\nu_1)>0,
$
contradicting \(P_{\nu_1}\lr P_{\nu_2}\).

\emph{Assertion~\textup{(ii)}.} Suppose \(P_{\nu_2}\lc P_{\nu_1}\) for all \(\nu_1\le\nu_2\), and assume by contradiction that \(K_{\nu_0}\) is not concave on \(J\) for some \(\nu_0 \in \mathcal I\). Then $s_{\nu_0}$ is not concave by~\eqref{eq:kernel-score-shape} applied to second differences. In the continuous case, there exist \(x_1<x_2<x_3\) in \(J\) and \(\lambda\in(0,1)\) satisfying \(x_2=\lambda x_1+(1-\lambda)x_3\) such that
\(s_{\nu_0}(x_2) <\lambda s_{\nu_0}(x_1)+(1-\lambda)s_{\nu_0}(x_3).\)
In the discrete case, the same display holds for an admissible triplet with \(\lambda=1/2\).
By continuity in \(\nu\), there is an \(\varepsilon>0\) and an interval $I_\varepsilon$ containing $\nu_0$ such that \(\lambda s_\nu(x_1)+(1-\lambda)s_\nu(x_3)-s_\nu(x_2)\ge\varepsilon\) for all \(\nu\in I_\varepsilon\). Choosing \(\nu_1<\nu_2\) in \(I_\varepsilon\) yields
\begin{alignat*}{2}
\lambda L_{\nu_2,\nu_1}(x_1) +(1-\lambda)L_{\nu_2,\nu_1}(x_3) -L_{\nu_2,\nu_1}(x_2)
&=\int_{\nu_1}^{\nu_2} \bigl[\lambda s_\nu(x_1)+(1-\lambda)s_\nu(x_3)-s_\nu(x_2)\bigr]\,\dx[\nu]\\
&\ge \varepsilon(\nu_2-\nu_1)>0,
\end{alignat*}
contradicting concavity of \(L_{\nu_2,\nu_1}\) on \(J\).

\emph{Assertion~\textup{(iii)}.} Suppose \(P_{\nu_1}\st P_{\nu_2}\) for all \(\nu_1\le\nu_2\) and assume \(\E[K_{\nu_0}(X)\mid X\ge x_0]-\E[K_{\nu_0}(X)]<0\) at some \((\nu_0,x_0)\) with \(\bar F_{\nu_0}(x_0)>0\). By~\eqref{eq:kernel-score-tail} and Lemma~\ref{lem:three-levels}, this is the value of $\partial_\nu\log\bar F_\nu(x_0)$ at $\nu=\nu_0$, which is continuous in $\nu$. Hence there is an \(\varepsilon>0\) and an interval \(I_\varepsilon\) containing \(\nu_0\) such that $\partial_\nu\log\bar F_\nu(x_0)<-\varepsilon$ for every $\nu\in I_\varepsilon$. Choosing \(\nu_1<\nu_2\) in \(I_\varepsilon\) gives $T_{\nu_2,\nu_1}(x_0)<0$, contradicting \(P_{\nu_1}\st P_{\nu_2}\).

\emph{Assertion~\textup{(iv)}.} Suppose \(P_{\nu_1}\hr P_{\nu_2}\) for all \(\nu_1\le\nu_2\) and assume \(K_{\nu_0}(x_0)-\E[K_{\nu_0}(X)\mid X\ge x_0]>0\) at some \((\nu_0,x_0)\) with \(\bar F_{\nu_0}(x_0)>0\). By~\eqref{def.log.companion} and~\eqref{eq:kernel-score-tail}, this difference equals
\[
s_{\nu_0}(x_0)-\E[s_{\nu_0}(X)\mid X\ge x_0]=\partial_\nu\log h_\nu(x_0)\Big|_{\nu=\nu_0}
\]
by Lemma~\ref{lem:three-levels}, which is continuous in $\nu$. Hence there is an \(\varepsilon>0\) and an interval \(I_\varepsilon\) containing \(\nu_0\) such that $\partial_\nu\log h_\nu(x_0)>\varepsilon$ for every $\nu\in I_\varepsilon$. Choosing \(\nu_1<\nu_2\) in \(I_\varepsilon\) then contradicts $H_{\nu_2,\nu_1}(x_0)\le 0$. This completes the proof.
\end{proof}

The simplest sufficient condition combines part \textup{(i)} of Theorem~\ref{thm:four-criteria} with the standard chain \(\lr\Rightarrow\hr\Rightarrow\st\)~\cite[Theorem 1.C.1]{ShakedShanthikumar}.

\begin{cor}
\label{cor:monotone-kernel}
If \(x\mapsto K_\nu(x)\), equivalently \(x\mapsto s_\nu(x)\), is nondecreasing for every \(\nu\in\mathcal I\), then for all \(\nu_1\le\nu_2\) in \(\mathcal I\),
\[
P_{\nu_1}\lr P_{\nu_2},\qquad
P_{\nu_1}\hr P_{\nu_2},\qquad
P_{\nu_1}\st P_{\nu_2}.
\]
If \(K_\nu\) is nonincreasing for every \(\nu\), the three orders reverse.
\end{cor}

\begin{proof}
Assertion \textup{(i)} of Theorem~\ref{thm:four-criteria} gives \(P_{\nu_1}\lr P_{\nu_2}\). The chain \(\lr\Rightarrow\hr\Rightarrow\st\) extends to the remaining orders.
\end{proof}

\subsection{Beyond monotone and concave kernels}
\label{subsec:beyond-monotone}

Theorem~\ref{thm:four-criteria}\,\textup{(i)} and \textup{(ii)} characterise \(\lr\) and \(\lc\) through monotonicity and concavity of \(K_\nu\) on \(J\). When these global shape properties fail, the tail-conditional criteria \textup{(iii)} and \textup{(iv)} can still yield \(\st\) and \(\hr\). The next results give shape conditions on the score \(s_\nu\) or kernel \(K_\nu\) that imply these tail-conditional inequalities. The superlevel-set, concave-kernel, and unimodal-kernel hypotheses below are kernel-level analogues of pairwise conditions on the likelihood ratio in~\cite{DerbaziPairwise}. The half-Student and zero-inflated examples of Section~\ref{sec:beyond-monotone-applications} illustrate each in turn.

\begin{prop}
\label{prop:kernel-superlevel}
Let \(\nu_1\le\nu_2\) in \(\mathcal I\), and assume \(J\) has finite left endpoint \(x_0\coloneqq \inf J\). If for every \(\nu\in[\nu_1,\nu_2]\), the set \(A_\nu\coloneqq \{x\in J:s_\nu(x)\ge 0\}\) is a nonempty initial interval of \(J\), then
\[
P_{\nu_2}\st P_{\nu_1}.
\]
If, in addition, \(s_\nu\) is nonincreasing on \(J\setminus A_\nu\) for every \(\nu\in[\nu_1,\nu_2]\), then
\[
P_{\nu_2}\hr P_{\nu_1}.
\]
\end{prop}

\begin{proof}
Fix \(\nu\in[\nu_1,\nu_2]\). Since \(A_\nu\) is a nonempty initial interval of \(J\), its complement \(J\setminus A_\nu\) is a final interval. Thus, \(s_\nu\) is nonnegative on \(A_\nu\) and negative on its complement. Consequently, \(s_\nu\) has at most one sign change on \(J\), from nonnegative to negative. Since \(\E[s_\nu(X)]=0\) for \(X\sim P_\nu\), Lemma~2.6 of~\cite{DerbaziPairwise}, applied with \(\phi=s_\nu\) under \(P_\nu\), gives
\[
\E[s_\nu(X)\mid X\ge x]\le 0
\]
for every \(x\in J\) with \(\bar F_\nu(x)>0\).

Now, by Lemma~\ref{lem:three-levels}, the left-hand side equals \(\partial_\nu\log\bar F_\nu(x)\). Integrating over \([\nu_1,\nu_2]\) gives \(\bar F_{\nu_2}(x)\le\bar F_{\nu_1}(x)\), which is \(P_{\nu_2}\st P_{\nu_1}\).

For the hazard-rate part, we need to show that \(s_\nu(x)\ge \E[s_\nu(X)\mid X\ge x]\) holds pointwise on \(\{x\in J:\bar F_\nu(x)>0\}\). If \(x\in A_\nu\), then the tail-mean bound proved in the first part gives \(\E[s_\nu(X)\mid X\ge x]\le0\le s_\nu(x)\). If \(x\in J\setminus A_\nu\), then every \(u\ge x\) also lies in \(J\setminus A_\nu\), and the right-tail monotonicity hypothesis gives \(s_\nu(u)\le s_\nu(x)\). In either case, \(\E[s_\nu(X)\mid X\ge x]\le s_\nu(x)\). By Lemma~\ref{lem:three-levels}, the difference equals \(\partial_\nu\log h_\nu(x)\), and integrating over \([\nu_1,\nu_2]\) gives \(h_{\nu_2}(x)\ge h_{\nu_1}(x)\), which is \(P_{\nu_2}\hr P_{\nu_1}\), as required.
\end{proof}

The superlevel-set hypothesis is the natural shape condition on \(s_\nu\) at fixed \(\nu\). Two corollaries follow, recovering kernel-level versions of the endpoint reductions of~\cite{DerbaziPairwise} under relative log-concavity and under unimodality of the likelihood ratio.

\begin{cor}
\label{cor:concave-kernel}
Let \(\nu_1\le\nu_2\) in \(\mathcal I\). Suppose \(K_\nu\) is concave on \(J\) and that \(J\) has finite left endpoint \(x_0\coloneqq \inf J\). If \(s_\nu(x_0)\ge 0\) for every \(\nu\in[\nu_1,\nu_2]\), then \(P_{\nu_2}\st P_{\nu_1}\) and \(P_{\nu_2}\hr P_{\nu_1}\).
\end{cor}

\begin{proof}
By definition, \(s_\nu\) inherits concavity from \(K_\nu\) and is centred under \(P_\nu\). A superlevel set of a concave function on an interval is an interval. Since \(s_\nu(x_0)\ge0\), the set \(A_\nu=\{s_\nu\ge 0\}\) is therefore a nonempty initial interval of \(J\). Concavity then implies that \(s_\nu\) is nonincreasing on \(J\setminus A_\nu\). Invoking Proposition~\ref{prop:kernel-superlevel} concludes the proof.
\end{proof}

\begin{cor}
\label{cor:unimodal-kernel}
Let \(\nu_1\le\nu_2\) in \(\mathcal I\), and suppose \(J\) has finite left endpoint \(x_0\coloneqq \inf J\). Suppose further that there exists a single point \(c\in J\), independent of \(\nu\), such that for every \(\nu\in[\nu_1,\nu_2]\):
\begin{enumerate}[label=\textup{(\roman*)}]
  \item \(K_\nu\) is nondecreasing on \(J\cap(-\infty,c]\).
  \item \(K_\nu\) is nonincreasing on \(J\cap[c,\infty)\).
  \item \(s_\nu(x_0)\ge 0\).
\end{enumerate}
Then \(P_{\nu_2}\st P_{\nu_1}\) and \(P_{\nu_2}\hr P_{\nu_1}\).
\end{cor}

\begin{proof}
The score is nondecreasing on \(J\cap(-\infty,c]\) and nonincreasing on \(J\cap[c,\infty)\), because subtracting a constant does not affect either monotonicity property. With \(s_\nu(x_0)\ge 0\) and \(\E[s_\nu]=0\), the set \(A_\nu=\{s_\nu\ge 0\}\) is a nonempty initial interval of \(J\). Therefore, \(s_\nu\) is nonincreasing on \(J\setminus A_\nu\subseteq J\cap[c,\infty)\). Now invoke Proposition~\ref{prop:kernel-superlevel} to complete the proof. 
\end{proof}

\subsection{Compound-law kernels}
\label{subsec:compound-constructions}

The compound laws considered here are discrete, so throughout this subsection we take \(E=\N_0\).
\begin{defn}[Compound laws]
\label{def:compound}
	Let $Q$ and $F$ be probability measures on $\N_0$ and $\N$, respectively.
  
  Define
\begin{equation}
  \label{def.compound-rv}
X=\sum_{i=1}^N J_i
\end{equation}
	where $N$ has law $Q$, and $(J_i)_{i\ge 1}$ is a sequence of i.i.d. random variables with law $F$, each independent of $N$. 
  
  Let $q_n \coloneqq Q(\{n\})$.  Conditioning on $N$ gives
	\[
	\prob(X=k)=\sum_{n\ge0} q_n\,F^{*n}(\{k\}), \quad k\in\N_0, 
	\]
where \(F^{*0}\coloneqq \delta_0\) and \(F^{*n}\coloneqq F^{*(n-1)}*F\) for \(n\ge1\).
\end{defn}

For a family of probability measures \((Q_\nu)_{\nu\in\mathcal I}\) on $\N_0$ and a fixed summand law $F$ on $\N$, write \(C_\nu\) for the law of the corresponding compound sum and \(f_\nu\coloneqq dC_\nu/d\mu\) for its mass function.

\begin{lem}
\label{lem:cond-mean-monotone}
Let $X$ be the compound sum defined in~\eqref{def.compound-rv}, and assume \(F\) is PF$_2$. Then the posterior kernel \((n,k)\mapsto\prob(N=n\mid X=k)\) is TP$_2$ on $\N_0\times\supp(X)$. In particular, the conditional laws of \(N\) given \(X=k\) are stochastically increasing in \(k\), and \(k\mapsto\E[N\mid X=k]\) is nondecreasing on \(\supp(X)\).
\end{lem}

\begin{proof}
Since \(F\) is PF$_2$, its convolution powers $M(n,k)\coloneqq F^{*n}(\{k\})$, $n,k\in\N_0$, form a TP$_2$ kernel~\cite{Karlin1968}. Multiplying the \(n\)th row of $M$ by \(q_n\) preserves the TP$_2$ property, and normalising each column preserves TP$_2$ as well. Hence the posterior kernel is TP$_2$, so the posterior laws of \(N\) given \(X=k\) are increasing in \(k\) in the monotone likelihood-ratio order, and therefore stochastically increasing. Their means are consequently nondecreasing.
\end{proof}

The compound law inherits a kernel from the counting law \(Q_\nu\). The general statement uses only the score of \(Q_\nu\) and produces the compound score. When \(Q_\nu\) is a member of the power-series family or admits a factorisation involving a parameter-dependent function, the same construction transports to a kernel.

\begin{prop}
	\label{prop:compound-score}
	Let \((Q_\nu)_{\nu\in\mathcal I}\) be a family of probability measures on \(\N_0\) with common support \(S\subseteq\N_0\), so that \(q_\nu(n)\coloneqq Q_\nu(\{n\})>0\) for \(n\in S\). Denote the density of the compound law \(C_\nu\) by
	\[
	f_\nu(k)\coloneqq \sum_{n\in S} q_\nu(n)F^{*n}(\{k\}).
	\]
If \(\nu\mapsto q_\nu(n)\) is differentiable on \(\mathcal I\) for every \(n\in S\), then for \(k\) in the common support of \(C_\nu\), the compound score is
	\begin{equation}
		\label{eq:compound-score}
		s_\nu(k)\coloneqq \partial_\nu\log f_\nu(k)=\E[\gamma_\nu(N)\mid X=k],
	\end{equation}
	where \(\gamma_\nu(n)\coloneqq \partial_\nu\log q_\nu(n)\) and the expectation is taken under the joint compound law at the current parameter value.
\end{prop}

\begin{proof}
For fixed \(k\), the sum defining \(f_\nu(k)\) is finite, because \(F^{*n}(\{k\})=0\) for \(n>k\). Since the atoms of \(Q_\nu\) are differentiable in \(\nu\), differentiating \(f_\nu(k)\) gives
\[
\partial_\nu f_\nu(k)=\sum_{n\in S}\gamma_\nu(n)q_\nu(n)F^{*n}(\{k\}).
\]
For \(k\) in the support of \(C_\nu\), we have
\[
\prob(N=n\mid X=k)
=\frac{q_\nu(n)F^{*n}(\{k\})}{f_\nu(k)}.
\]
Dividing the derivative identity by \(f_\nu(k)\) yields
\begin{alignat*}{2}
\partial_\nu\log f_\nu(k)=\sum_{n\in S}\gamma_\nu(n)\frac{q_\nu(n)F^{*n}(\{k\})}{f_\nu(k)}
&=\sum_{n\in S}\gamma_\nu(n)\prob(N=n\mid X=k)\\&=\E[\gamma_\nu(N)\mid X=k].
\end{alignat*}
\end{proof}

\begin{cor}
\label{cor:compound-kernel}

Suppose the hypotheses of Proposition~\ref{prop:compound-score} hold and that \(q_\nu(n)=w_n(\nu)/A(\nu)\), where \(w_n(\nu)>0\) for \(n\in S\), \(A(\nu)=\sum_{n\in S}w_n(\nu)\), and \(\partial_\nu A(\nu)=\sum_{n\in S}\partial_\nu w_n(\nu)\). Set \(\mathcal G_\nu(n)\coloneqq \partial_\nu\log w_n(\nu)\). Then
\[
K_\nu(k) \coloneqq \E[\mathcal G_\nu(N) \mid X = k]
\]
is a kernel of $C_\nu$.
\end{cor}

\begin{proof}
Since \(q_\nu(n)=w_n(\nu)/A(\nu)\), the score of $Q_\nu$ is $\gamma_\nu(n)=\mathcal G_\nu(n)-\partial_\nu\log A(\nu).$ Proposition~\ref{prop:compound-score} therefore implies that
\[
s_\nu(k)=\E[\gamma_\nu(N)\mid X=k]
=\E[\mathcal G_\nu(N)\mid X=k]-\partial_\nu\log A(\nu)
=K_\nu(k)-\partial_\nu\log A(\nu).
\]
The differentiability condition on \(A\) gives
\[
\partial_\nu\log A(\nu)=\sum_{n\in S}\mathcal G_\nu(n)q_\nu(n)=\E[\mathcal G_\nu(N)]=\E[K_\nu(X)].
\]
Hence \(K_\nu\) is a kernel of \(C_\nu\).
\end{proof}

The compound kernel \(K_\nu\) is the posterior expectation of \(\mathcal G_\nu\). Since centring does not affect monotonicity, the compound score and compound kernel have the same monotonicity properties in \(k\). Combining Lemma~\ref{lem:cond-mean-monotone} with Proposition~\ref{prop:compound-score} or Corollary~\ref{cor:compound-kernel} gives a likelihood-ratio criterion for the compound law in terms of the score or kernel of \(Q_\nu\) alone.

\begin{prop}
\label{prop:compound-monotone}
Assume the summand law \(F\) is PF$_2$, and that the hypotheses of Corollary~\ref{cor:compound-kernel} hold.
\begin{enumerate}[label=\textup{(\roman*)}]
\item If \(n\mapsto\mathcal G_\nu(n)\) is nondecreasing on \(S\) for every \(\nu\in\mathcal I\), then \(C_{\nu_1}\lr C_{\nu_2}\) for \(\nu_1\le\nu_2\).
\item If \(n\mapsto\mathcal G_\nu(n)\) is nonincreasing on \(S\) for every \(\nu\in\mathcal I\), then \(C_{\nu_2}\lr C_{\nu_1}\) for \(\nu_1\le\nu_2\).
\end{enumerate}
\end{prop}

\begin{proof}
By Lemma~\ref{lem:cond-mean-monotone}, the conditional laws of \(N\) given \(X=k\) are stochastically increasing in \(k\). Hence, the conditional expectation of any nondecreasing function of \(N\) is nondecreasing in \(k\)~\cite[Theorem 1.C.5]{ShakedShanthikumar}. By Corollary~\ref{cor:compound-kernel}, \(K_\nu(k)=\E[\mathcal G_\nu(N)\mid X=k]\), which is therefore nondecreasing in \(k\) under hypothesis~\textup{(i)} and nonincreasing in \(k\) under~\textup{(ii)}.

For \textup{(i)}, fix \(k_1\le k_2\) in the common support of \(C_\nu\), and set
\[
G_{k_1,k_2}(\nu)\coloneqq \log f_\nu(k_2)-\log f_\nu(k_1).
\]
Since \(K_\nu\) is nondecreasing in \(k\) and \(s_\nu=K_\nu-\E[K_\nu(X)]\), Proposition~\ref{prop:compound-score} gives
\[
G'_{k_1,k_2}(\nu)=s_\nu(k_2)-s_\nu(k_1)\ge 0.
\]
Thus \(G_{k_1,k_2}\) is nondecreasing in \(\nu\). For \(\nu_1\le\nu_2\),
\[
\log\frac{f_{\nu_2}(k_2)}{f_{\nu_1}(k_2)}
-\log\frac{f_{\nu_2}(k_1)}{f_{\nu_1}(k_1)}
=G_{k_1,k_2}(\nu_2)-G_{k_1,k_2}(\nu_1)\ge 0.
\]
Hence \(f_{\nu_2}/f_{\nu_1}\) is nondecreasing, equivalently \(f_{\nu_1}/f_{\nu_2}\) is nonincreasing, so \(C_{\nu_1}\lr C_{\nu_2}\). The proof of \textup{(ii)} is identical with all inequalities reversed, giving \(C_{\nu_2}\lr C_{\nu_1}\).
\end{proof}

\begin{rem}
	\label{rem.affine.harmonic.kernel}
When \(\mathcal G_\nu\) is affine, \(\mathcal G_\nu(n)=a(\nu)+b(\nu)n\), Corollary~\ref{cor:compound-kernel} and Lemma~\ref{lem:cond-mean-monotone} give the explicit form
\[
K_\nu(k)=a(\nu)+b(\nu)\E[N\mid X=k],
\]
and monotonicity of \(\mathcal G_\nu\) reduces to the sign of \(b(\nu)\). The five classical counting laws fall under this case (Section~\ref{sec:compound-applications}). Harmonic kernels such as \(\psi(\alpha+n)-\psi(\alpha)\), which are nondecreasing in \(n\) without being affine, fall under Proposition~\ref{prop:compound-monotone} directly.
\end{rem}


\section{Applications}
\label{sec:applications}

The table below is organised by the way the varying parameter enters the density or mass function.

\begin{table}[!ht]
\renewcommand{\arraystretch}{1.4}
\centering
\footnotesize
\caption{Common kernels by parameter role.}
\label{tab:common-kernels}
\begin{tabular}{@{}p{5.0cm}p{2.8cm}p{2.6cm}p{3.4cm}@{}}
\toprule
\textbf{Representation / parameter} & \textbf{Kernel} & $\boldsymbol{K'(x),\ \Delta K(k)}$ & $\boldsymbol{K''(x),\ \Delta^2 K(k)}$ \\
\midrule
\multicolumn{4}{@{}l}{\emph{Natural parameter} (exponential family with statistic \(T\))} \\
\multicolumn{4}{@{}l}{\textit{\,continuous}} \\
Statistic \(T(x)\) & $T(x)$ & $T'(x)$ & $T''(x)$ \\
\multicolumn{4}{@{}l}{\textit{\,discrete}} \\
Statistic \(T(k)\) & $T(k)$ & $\Delta T(k)$ & $\Delta^2 T(k)$ \\
Conway--Maxwell--Poisson (CMP) dispersion, \(T(k)=-\log k!\) & $-\log k!$ & $-\log(k+1)<0$ & $-\log\frac{k+2}{k+1}<0$ \\
\midrule
\multicolumn{4}{@{}l}{\emph{Power, rate, or scale parameter}} \\
\multicolumn{4}{@{}l}{\textit{\,continuous}} \\
Exponential, Gamma rate \(\rho\) & $-x$ & $-1$ & $0$ \\
Weibull rate \(\lambda\), fixed shape \(\beta\) & $-x^\beta$ & $-\beta x^{\beta-1}$ & $-\beta(\beta-1)x^{\beta-2}$ \\
Half-normal scale \(\sigma\) & $x^2/\sigma^3$ & $2x/\sigma^3$ & $2/\sigma^3>0$ \\
\multicolumn{4}{@{}l}{\textit{\,discrete}} \\
Power-series factor \(a(k)\theta^k\) & $k/\theta$ & $1/\theta>0$ & $0$ \\
\midrule
\multicolumn{4}{@{}l}{\emph{Shape parameter}} \\
\multicolumn{4}{@{}l}{\textit{\,continuous}} \\
Gamma shape, Beta first shape & $\log x$ & $1/x$ & $-1/x^2$ \\
Pareto shape, fixed scale & $-\log x$ & $-1/x$ & $1/x^2$ \\
Beta second shape & $\log(1-x)$ & $-1/(1-x)$ & $-1/(1-x)^2$ \\
\multicolumn{4}{@{}l}{\textit{\,discrete}} \\
Upper Pochhammer \((\nu)_k\) & $\psi(\nu+k)-\psi(\nu)$ & $\frac{1}{\nu+k}>0$ & $-\frac{1}{(\nu+k)(\nu+k+1)}<0$ \\
Lower Pochhammer \(1/(\nu)_k\) & $-[\psi(\nu+k)-\psi(\nu)]$ & $-\frac{1}{\nu+k}<0$ & $\frac{1}{(\nu+k)(\nu+k+1)}>0$ \\
Finite-support \((\nu)_{n-k}\) & $\psi(\nu+n-k)-\psi(\nu)$ & $-\frac{1}{\nu+n-k-1}<0$ & $-\frac{1}{(\nu+n-k-2)(\nu+n-k-1)}<0$ \\
\midrule
\multicolumn{4}{@{}l}{\emph{Location parameter}} \\
\multicolumn{4}{@{}l}{\textit{\,continuous}} \\
Translate \(f_0(x-\mu)\) & $-(\log f_0)'(x-\mu)$ & $-(\log f_0)''(x-\mu)$ & $-(\log f_0)'''(x-\mu)$ \\
Gumbel, \(f_0(x)=\exp\{-x-e^{-x}\}\) & $-e^{-x}$ & $e^{-x}$ & $-e^{-x}$ \\
Log-normal location \(\mu\) (location on \(\log\)-scale, fixed \(\sigma\)) & $(\log x-\mu)/\sigma^2$ & $1/(\sigma^2 x)$ & $-1/(\sigma^2 x^2)$ \\
\multicolumn{4}{@{}l}{\textit{\,discrete}} \\
Translate \(f_0(k-\mu)\) on \(\Z\) & $-\Delta\log f_0(k-\mu)$ & $-\Delta^2\log f_0(k-\mu)$ & $-\Delta^3\log f_0(k-\mu)$ \\
\bottomrule
\end{tabular}
\end{table}

\subsection{Ordering in parametric density families}
\label{sec:within-applications}
To make the relation between score and kernel explicit, each template writes \(f_\nu=w_\nu/Z(\nu)\) and reads off the score \(s_\nu=\partial_\nu\log f_\nu\) and the kernel \(K_\nu=\partial_\nu\log w_\nu\).
\begin{enumerate}[label=\textup{(\roman*)}]
  \item \emph{Natural parameter} \(\nu\), through \(w_\nu(x)=h(x)\exp\{\nu T(x)\}\) and \(Z(\nu)=e^{A(\nu)}\):
  \[
  s_\nu(x)=T(x)-A'(\nu),\qquad K_\nu(x)=T(x),
  \]
  the sufficient statistic. The normalisation term \(A'(\nu)=\E[T(X)]\) is the centring constant.

  \item \emph{Power, rate, or scale parameter}. In the discrete case the parameter-dependent factor is \(w_\theta(k)=a(k)\theta^k\) with \(Z(\theta)=\sum_k a(k)\theta^k\), so
  \[
  s_\theta(k)=\frac{k-\E[X]}{\theta},\qquad K_\theta(k)=\frac{k}{\theta},
  \]
  the affine power-series kernel. In the continuous rate case the parameter-dependent factor is \(w_\rho(x)=h(x)e^{-\rho R(x)}\) with \(Z(\rho)=\int h\,e^{-\rho R}\,\dx\), so
  \[
  s_\rho(x)=-R(x)+\E[R(X)],\qquad K_\rho(x)=-R(x).
  \]
  In the continuous scale case \(f_\sigma(x)=\sigma^{-1}f_0(x/\sigma)\) on \((0,\infty)\), take \(w_\sigma(x)=f_0(x/\sigma)\), with the \(\sigma^{-1}\) prefactor represented by \(Z(\sigma)=\sigma\). Then
  \[
  s_\sigma(x)=-\tfrac{1}{\sigma}-\tfrac{x}{\sigma^2}(\log f_0)'(x/\sigma),\qquad K_\sigma(x)=-\tfrac{x}{\sigma^2}(\log f_0)'(x/\sigma).
  \]

  \item \emph{Shape parameter} \(\nu\), through a shape-dependent factor \(w_\nu\), for instance an upper Pochhammer block \((\nu)_k\):
  \[
  s_\nu(x)=\partial_\nu\log w_\nu(x)-\E[\partial_\nu\log w_\nu(X)],\qquad K_\nu(x)=\partial_\nu\log w_\nu(x).
  \]
  When \(\nu\) enters through an upper Pochhammer block, the kernel evaluates to \(\psi(\nu+k)-\psi(\nu)\), with lower and finite-support variants in Table~\ref{tab:common-kernels}. The continuous analogue is the logarithmic kernel \(\log x\).

  \item \emph{Location parameter} \(\mu\), through the translate \(w_\mu(x)=f_0(x-\mu)\). The translate has \(Z(\mu)=1\) independent of \(\mu\), so the normalisation term vanishes and score and kernel coincide:
  \[
  s_\mu(x)=K_\mu(x)=-(\log f_0)'(x-\mu).
  \]
  The discrete case replaces the derivative by the forward difference of \(\log f_0\).
\end{enumerate}
The examples below use the entries of Table~\ref{tab:common-kernels}. The simplest applications use Corollary~\ref{cor:monotone-kernel}: monotonicity of the kernel gives $\lr$, $\hr$, and $\st$, and concavity of the kernel gives $\lc$ through Theorem~\ref{thm:four-criteria}.

\begin{example}[Standard discrete orderings]
\label{ex:standard-discrete}
The Poisson, geometric, and negative-binomial laws are generalised power-series in the power parameter, with coefficient sequences \(a_k=1/k!\), \(a_k=1\), and \(a_k=(r)_k/k!\). By the power-parameter row of Table~\ref{tab:common-kernels},
\[
\Poi(\theta_1)\lr\Poi(\theta_2), \qquad \Geom(p_1)\lr\Geom(p_2), \qquad \NB(\alpha,p_1)\lr\NB(\alpha,p_2)
\]
for $\theta_1\le\theta_2$ and $p_2\le p_1$ (so that \(q=1-p\) is nondecreasing).

In the shape parameter, $\NB(\nu,p)$ and the beta-binomial $\BetaBin(n,r,s)$ in the upper shape \(r\) enter through an upper Pochhammer block. By the shape-parameter row of Table~\ref{tab:common-kernels} (upper Pochhammer) and Theorem~\ref{thm:four-criteria},
\[
\NB(\nu_1,p)\lr\NB(\nu_2,p)
\quad\text{and}\quad
\NB(\nu_2,p)\lc\NB(\nu_1,p),
\qquad \nu_1\le\nu_2,
\]
and the analogous statements hold for $\BetaBin$ in its first shape parameter $r$. The lower shape \(s\) of $\BetaBin$ enters through a finite-support Pochhammer block, whose row gives the reverse \(\lr\) direction.
\end{example}
\begin{example}[Standard continuous orderings]
\label{ex:standard-continuous}
Reading off the continuous rows of Table~\ref{tab:common-kernels}:
\begin{itemize}
\item \emph{Gamma.} The shape-parameter row gives \(\Gam(r_1,\rho)\lr\Gam(r_2,\rho)\) for \(r_1\le r_2\). The rate-parameter row has nonincreasing kernel \(-x\), so the order reverses: \(\Gam(r,\rho_2)\lr\Gam(r,\rho_1)\) for \(\rho_1\le\rho_2\). The exponential case is \(r=1\).

\item \emph{Beta.} The shape-parameter rows for \(\log x\) and \(\log(1-x)\) give \(\Bet(\alpha_1,\beta)\lr\Bet(\alpha_2,\beta)\) for \(\alpha_1\le\alpha_2\) and the reverse direction in \(\beta\).

\item \emph{Pareto.} The shape-parameter row for \(-\log x\) is nonincreasing, so \(\Par(\alpha_2)\lr\Par(\alpha_1)\) whenever \(\alpha_1\le\alpha_2\), with the scale parameter fixed.

\item \emph{Half-normal.} The scale-parameter row gives \(\HN(\sigma_1)\lr\HN(\sigma_2)\) for \(\sigma_1\le\sigma_2\).

\item \emph{Log-normal.} The location-parameter row gives \(\LN(\mu_1,\sigma)\lr\LN(\mu_2,\sigma)\) for \(\mu_1\le\mu_2\).
\end{itemize}
\end{example}

\subsection{Beyond monotone-kernel applications}
\label{sec:beyond-monotone-applications}

When the kernel is neither monotone nor concave on the full support, the \(\lr\) and \(\lc\) routes are unavailable but the criteria of Section~\ref{subsec:beyond-monotone} can still deliver \(\hr\) and \(\st\). Two patterns recur. The half-line pattern restricts to \([0,\infty)\) or to an interval where the kernel is unimodal with a positive value at the boundary, bringing Corollary~\ref{cor:unimodal-kernel} into reach. The tail-conditional pattern verifies the inequalities of Theorem~\ref{thm:four-criteria}\,\textup{(iii)}--\textup{(iv)} directly, comparing the kernel value at \(m\) with its conditional expectation on \(\{X\ge m\}\). This pattern fits zero-inflated laws, where atom-inflation breaks monotonicity of the kernel at the origin while the tail-conditional inequalities remain explicit.

\begin{example}[Half-Student in degrees of freedom]
\label{ex:half-student-family}
Let \((P_\nu)_{\nu>0}\) be the half-Student family on \(E=[0,\infty)\), with density
\[
f_\nu(x) =\frac{2\,\Gamma((\nu+1)/2)}{\sqrt{\nu\pi}\,\Gamma(\nu/2)} \Bigl(1+\tfrac{x^2}{\nu}\Bigr)^{-(\nu+1)/2}, \qquad x\ge 0.
\]
A direct calculation gives
\[
(s_\nu)'(x)=\frac{x(1-x^2)}{(\nu+x^2)^2}, \qquad x>0,
\]
so \(s_\nu\) is unimodal in \(x\) on \([0,\infty)\) with common mode \(c=1\) for every \(\nu>0\). Writing \(z=\nu/2\),
\[
s_\nu(0)=\partial_\nu\log f_\nu(0) =\frac12\{\psi(z+\tfrac12)-\psi(z)\}-\frac{1}{4z}>0,
\]
because concavity of \(\psi\) and \(\psi(z+1)-\psi(z)=1/z\) imply \(\psi(z+\tfrac12)-\psi(z)>1/(2z)\). Both hypotheses of Corollary~\ref{cor:unimodal-kernel} (common-mode unimodality and \(s_\nu(x_0)>0\)) hold uniformly in \(\nu\), so
\[
\nu_1<\nu_2 \quad\Longrightarrow\quad P_{\nu_2}\hr P_{\nu_1} \text{ and } P_{\nu_2}\st P_{\nu_1}.
\]
On \(\R\) the underlying Student-\(t\) kernel is symmetric and neither monotone nor concave, so the global criteria of Theorem~\ref{thm:four-criteria} do not apply. 
\end{example}

\begin{example}[Zero-inflated Poisson]
\label{ex:zero-inflated}
Take the Poisson base law $Q_\theta=\Poi(\theta)$ with $\theta>0$, and use the superscript $Q$ to mark its kernel and score. A kernel of $Q_\theta$ is the affine GPS kernel $K_\theta^Q(k)=k/\theta$, nondecreasing in $k$. The Poisson family is therefore $\lr$-, $\hr$-, and $\st$-monotone in $\theta$. Its score is \(s_\theta^Q(k)=k/\theta-1\), so \(c_Q(\theta)=K_\theta^Q(k)-s_\theta^Q(k)=1\). Zero-inflation gives, for $\pi\in(0,1)$,
\[
f_\theta(0)=(1-\pi)+\pi e^{-\theta},
\qquad
f_\theta(k)=\pi e^{-\theta}\theta^k/k!\quad (k\ge 1).
\]
Differentiating in $\theta$ gives the score of the zero-inflated law. Adding the constant $c_Q(\theta)=1$ produces a kernel that agrees with $K_\theta^Q$ at every $k\ge 1$:
\[
K_\theta(k)=k/\theta\quad (k\ge 1),
\qquad
K_\theta(0)=1-\alpha_\theta,
\qquad
\alpha_\theta\coloneqq \frac{\pi e^{-\theta}}{(1-\pi)+\pi e^{-\theta}}.
\]
Now $1-\alpha_\theta=(1-\pi)e^{\theta}/[(1-\pi)e^{\theta}+\pi]$, so as $\theta\to\infty$, \(1-\alpha_\theta\to 1\) and $K_\theta(0)\to 1$, while $K_\theta(1)=1/\theta\to 0$. For large $\theta$, $K_\theta(0)>K_\theta(1)$, so $\Delta K_\theta$ changes sign and the zero-inflated family is not $\lr$-ordered in $\theta$. Similarly $\Delta^2 K_\theta(0)>0$ rules out $\lc$. The \(\st\) and \(\hr\) directions can nevertheless be checked directly from Theorem~\ref{thm:four-criteria}\,\textup{(iii)}--(iv). Since \(K_\theta\) is a kernel, \(\E[K_\theta(X)]=1\). At \(m=0\), \(\E[K_\theta(X)\mid X\ge 0]=1\). For \(m\ge 1\), if \(N\sim\Poi(\theta)\), the event \(\{X\ge m\}\) excludes the inflated atom, so
\[
\E[K_\theta(X)\mid X\ge m]
=\frac{1}{\theta}\E[N\mid N\ge m]
=\frac{\prob(N\ge m-1)}{\prob(N\ge m)}
\ge 1.
\]
Thus condition \textup{(iii)} holds. Moreover,
\[
K_\theta(m)=\frac{m}{\theta}
\le \frac{1}{\theta}\E[N\mid N\ge m]
=\E[K_\theta(X)\mid X\ge m]
\qquad (m\ge 1),
\]
while \(K_\theta(0)=1-\alpha_\theta\le 1\), so condition \textup{(iv)} holds as well. Therefore \(P_{\theta_1}\st P_{\theta_2}\) and \(P_{\theta_1}\hr P_{\theta_2}\) for all \(\theta_1\le\theta_2\), by Theorem~\ref{thm:four-criteria}\,\textup{(iii)}--\textup{(iv)}.
\end{example}

\begin{example}[Zero-inflated exponential]
\label{ex:zero-inflated-exp}
The continuous analogue takes the base law \(Q_\theta=\Exp(\theta)\) on \((0,\infty)\) and forms the atom-inflated law
\[
P_\theta=(1-\pi)\delta_0+\pi \Exp(\theta)
\qquad\text{on }[0,\infty).
\]
With respect to the dominating measure \(\delta_0+\dx\), the score is
\[
s_\theta(0)=0,\qquad
s_\theta(x)=\frac{1}{\theta}-x,\quad x>0.
\]
Since \(\E[s_\theta(X)]=0\), the score \(s_\theta\) is itself a kernel, and we write \(K_\theta\coloneqq s_\theta\). For \(m>0\), the event \(\{X\ge m\}\) excludes the atom at \(0\), so the conditional law coincides with the exponential tail. Hence
\begin{align*}
\E[K_\theta(X)\mid X\ge m]
&=-m\le 0=\E[K_\theta(X)],\\
K_\theta(m)
&=\frac{1}{\theta}-m\ge -m
=\E[K_\theta(X)\mid X\ge m].
\end{align*}
At \(m=0\), both inequalities hold with equality. Thus the inequalities of Theorem~\ref{thm:four-criteria}\,\textup{(iii)}--\textup{(iv)} hold with the reversed sign, so the family is \(\st\)- and \(\hr\)-decreasing in \(\theta\): \(P_{\theta_2}\st P_{\theta_1}\) and \(P_{\theta_2}\hr P_{\theta_1}\) whenever \(\theta_1\le\theta_2\).

The likelihood-ratio direction still fails. With respect to the dominating measure \(\delta_0+\dx\), the density of \(P_\theta\) is
\[
\frac{\dx[]P_\theta}{\dx[\delta_0+\dx]}(x)
=\begin{dcases}
1-\pi, & x=0,\\
\pi\theta e^{-\theta x}, & x>0.
\end{dcases}
\]
For \(\theta_1<\theta_2\), the corresponding likelihood ratio equals
\[
1 \quad\text{at }x=0,
\qquad
\frac{\theta_2}{\theta_1}e^{-(\theta_2-\theta_1)x}
\quad\text{for }x>0.
\]
Since \(\theta_2/\theta_1>1\), the likelihood ratio jumps upward at \(0\), so it is not nonincreasing on \([0,\infty)\). Therefore the family is not \(\lr\)-ordered in \(\theta\), even though the \(\st\)- and \(\hr\)-orders survive atom inflation.
\end{example}

\subsection{Joint-parameter and interpolation paths}
\label{sec:path-applications}

The density-family criteria apply directly to any \(C^1\) path \(t\mapsto P_t\) with common support and satisfying Assumption~\ref{asm:diff-under-integral}, with \(t\) in place of \(\nu\). Treat \(t\) as the varying parameter and use the path score \(\partial_t\log f_t\). When a reduced factor \(w_t\) is available, the path kernel is \(K_t(x)=\partial_t\log w_t(x)\).

For a path \(t\mapsto(\theta_1(t),\dots,\theta_d(t))\) inside a $d$-parameter family with parameter-dependent factor \(w_{\theta_1,\dots,\theta_d}\), the chain rule gives the path kernel as a weighted combination of single-parameter kernels:
\begin{equation}\label{eq:chain-rule-kernel}
K_t(x)=\sum_{i=1}^{d}\theta_i'(t)\,K^{(i)}_{\theta(t)}(x),
\qquad K^{(i)}_{\theta(t)}(x)\coloneqq
\left.\partial_{\theta_i}\log w_{\theta_1,\dots,\theta_d}(x)\right|_{(\theta_1,\dots,\theta_d)=\theta(t)}.
\end{equation}
The joint-parameter examples (Examples~\ref{ex:nb-joint}, \ref{ex:bb-joint-rs}, and~\ref{ex:gamma-path}) read off the single-parameter kernels from Table~\ref{tab:common-kernels} and combine them through~\eqref{eq:chain-rule-kernel}. The interpolation example (Example~\ref{ex:bb-bin}) connects two laws of distinct factor form along a constructed path on which the kernel varies nontrivially with the path parameter. The degenerate case, in which the path kernel is constant along the path, is treated separately in Section~\ref{sec:pairwise-applications}.

\begin{example}[Negative binomial along a two-parameter path]
\label{ex:nb-joint}
Let \(r_1\le r_2\) and \(q_1\le q_2\) with \(q_i=1-p_i\). Along the linear path \(r(t)=(1-t)r_1+tr_2\), \(q(t)=(1-t)q_1+tq_2\), the chain rule~\eqref{eq:chain-rule-kernel} combines the harmonic kernel \(\psi(r+k)-\psi(r)\) in the shape parameter (upper Pochhammer row of Table~\ref{tab:common-kernels}) with the affine kernel \(k/q\) in the success-failure parameter \(q=1-p\) (discrete power-series row):
\[
K_t(k)
=r'(t)\,\bigl[\psi(r(t)+k)-\psi(r(t))\bigr]
+q'(t)\,\frac{k}{q(t)},
\]
nondecreasing and concave in \(k\) for \(r'(t),q'(t)\ge 0\). Theorem~\ref{thm:four-criteria} applied along the path gives \(\NB(r_1,p_1)\lr\NB(r_2,p_2)\) and \(\NB(r_2,p_2)\lc\NB(r_1,p_1)\).
\end{example}

\begin{example}[Beta-binomial along a two-parameter path]
\label{ex:bb-joint-rs}
Let \(r_1\le r_2\) and \(s_1\ge s_2\), with all parameters positive and \(n\ge 1\) fixed. The Beta-binomial mass function has parameter-dependent factor \(w_{r,s}(k)\propto(r)_k(s)_{n-k}\), with single-parameter kernels \(\psi(r+k)-\psi(r)\) (upper Pochhammer) and \(\psi(s+n-k)-\psi(s)\) (finite-support Pochhammer) from Table~\ref{tab:common-kernels}. Along the linear path \(r(t)=(1-t)r_1+tr_2\), \(s(t)=(1-t)s_1+ts_2\), the chain rule~\eqref{eq:chain-rule-kernel} gives a path kernel with forward difference
\[
\Delta K_t(k)
=\frac{r'(t)}{r(t)+k}-\frac{s'(t)}{s(t)+n-k-1}\ge 0,
\qquad k=0,\dots,n-1,
\]
because \(r'(t)\ge 0\) and \(s'(t)\le 0\). Theorem~\ref{thm:four-criteria} applied along the path gives \(\BetaBin(n,r_1,s_1)\lr\BetaBin(n,r_2,s_2)\).
\end{example}

\begin{example}[Beta-binomial to binomial interpolation]
\label{ex:bb-bin}
If $p\ge (r+n-1)/(r+s+n-1)$, then $\BetaBin(n,r,s)\lr \Bin(n,p)$.

Consider the path of factors \(w_c(k)=\binom{n}{k}(r+cp)_k(s+c(1-p))_{n-k}\), \(c\ge 0\), and let \(P_c\) be the law obtained by normalising \(w_c\). At \(c=0\) this is \(\BetaBin(n,r,s)\). As \(c\to\infty\), \((a+cq)_k=(cq)^k(1+O(1/c))\) uniformly in \(k\le n\), so \(w_c(k)/c^n\to\binom{n}{k}p^k(1-p)^{n-k}\) pointwise in \(k\), and \(P_c\to\Bin(n,p)\) after normalisation. The path kernel is \(K_c(k)=\partial_c\log w_c(k)\), so differentiating each Pochhammer factor through \(\partial_c\log(r+cp)_k=p[\psi(r+cp+k)-\psi(r+cp)]\) and the analogous identity for the second factor gives
\[
K_c(k)=p\bigl[\psi(r+cp+k)-\psi(r+cp)\bigr]+(1-p)\bigl[\psi(s+c(1-p)+n-k)-\psi(s+c(1-p))\bigr].
\]
Its forward difference is
\[
\Delta K_c(k)
=\frac{p}{r+cp+k}-\frac{1-p}{s+c(1-p)+n-k-1}
=\frac{p(s+n-1)-(1-p)r-k}{(r+cp+k)(s+c(1-p)+n-k-1)}.
\]
The stated condition \(p\ge(r+n-1)/(r+s+n-1)\) is exactly the requirement that the numerator be nonnegative at \(k=n-1\), so \(\Delta K_c(k)\ge 0\) for every \(k=0,\dots,n-1\). Theorem~\ref{thm:four-criteria}\,\textup{(i)} gives \(P_0\lr P_c\) for every \(c>0\). The likelihood-ratio order is preserved under pointwise limits of densities, so letting \(c\to\infty\) yields \(\BetaBin(n,r,s)\lr\Bin(n,p)\).
\end{example}

\begin{example}[Gamma along a shape--scale path]
\label{ex:gamma-path}
For a continuous analogue, take a \(C^1\) path \(t\mapsto(r(t),\beta(t))\) with \(r'(t)\ge 0\) and \(\beta'(t)\ge 0\). The gamma density \(f_{r,\beta}(x)\propto x^{r-1}e^{-x/\beta}\) has parameter-dependent factor \(w_{r,\beta}(x)=x^{r-1}e^{-x/\beta}\), with single-parameter kernels
\[
K^{(r)}_{r}(x)=\partial_r\log w_{r,\beta}(x)=\log x,
\qquad
K^{(\beta)}_{\beta}(x)=\partial_\beta\log w_{r,\beta}(x)=x/\beta^2.
\]
Both are nondecreasing on \((0,\infty)\). The first is the shape-parameter row of Table~\ref{tab:common-kernels}; the second comes from differentiating \(-x/\beta\) in \(\beta\). The chain-rule combination~\eqref{eq:chain-rule-kernel} gives the path kernel
\[
K_t(x)=r'(t)\log x+\beta'(t)\,x/\beta(t)^2,
\]
which is once more nondecreasing on \((0,\infty)\). Theorem~\ref{thm:four-criteria}\,\textup{(i)} therefore gives
\[
\Gam(r_0,\beta_0)\lr\Gam(r_1,\beta_1) \qquad\text{whenever } r_0\le r_1 \text{ and } \beta_0\le\beta_1,
\]
on taking any monotone path with \(r'\ge0\) and \(\beta'\ge0\) between the endpoints, for instance the linear path. In the rate parametrisation \(\rho=1/\beta\), the scale kernel \(x/\beta^2\) becomes the rate kernel \(-x\) (the rate-parameter row of Table~\ref{tab:common-kernels}, up to sign), and a path with \(r'(t)\ge0\) and \(\rho'(t)\le0\) has kernel
\[
K_t(x)=r'(t)\log x-\rho'(t)\,x,
\]
again nondecreasing. So increasing the shape \(r\) and decreasing the rate \(\rho\) both push the gamma law upward in \(\lr\). Consequently, for \(r_0\le r_1\) and \(\rho_0\ge\rho_1\),
\[
\Gam(r_0,\rho_0)\lr\Gam(r_1,\rho_1).
\]

\end{example}

\subsection{Comparisons via the pairwise kernel}
\label{sec:pairwise-applications}

The joint-parameter comparisons of the previous subsection use an explicit one-parameter path. When two laws admit compatible factorisations, the algebraic log-ratio of their factors can be used directly, without choosing a cumbersome and application-specific path. Fix two laws \(P,Q\) on a common discrete support \(J\subseteq\N_0\) with positive factors \(w^P_k,w^Q_k\), so \(f_P(k)\propto w^P_k\) and \(f_Q(k)\propto w^Q_k\). The path \(w_t(k)=(w^P_k)^t(w^Q_k)^{1-t}\), \(t\in[0,1]\), interpolates from \(Q\) at \(t=0\) to \(P\) at \(t=1\), and its path score is constant in \(t\) and equal to the log-factor ratio:
\[
\partial_t\log w_t(k) =\log\frac{w^P_k}{w^Q_k}.
\]
The \emph{pairwise kernel}
\[
{\cal K}(k)\coloneqq \log\frac{w^P_k}{w^Q_k}, \qquad k\in J,
\]
is therefore a path kernel along this interpolation, defined up to an additive constant absorbed by the normalising ratio. Theorem~\ref{thm:four-criteria} applied along the path with \(t_0=0\) and \(t_1=1\) reads
\[
Q\lr P \iff {\cal K}\text{ is nondecreasing on }J, \qquad P\lc Q \iff {\cal K}\text{ is concave on }J,
\]
with the discrete tests \(\Delta{\cal K}(k)\ge 0\) and \(\Delta^2{\cal K}(k)\le 0\). Under this concavity condition, the hazard-rate and usual stochastic orders reduce to endpoint checks at \(k_*=\min J\) by~\cite{DerbaziPairwise}. Each pairing of factor forms below decomposes \({\cal K}\) into transparent blocks.

\paragraph{Shared GPS representation.}
For \(w^P_k=a_k\theta_P^k\) and \(w^Q_k=b_k\theta_Q^k\),
\[
{\cal K}(k) =\log\frac{a_k}{b_k}+k\log\frac{\theta_P}{\theta_Q}, \qquad \Delta^2{\cal K}(k) =\Delta^2\log\frac{a_k}{b_k}.
\]
The affine power term shifts \(\Delta{\cal K}\) but drops out of \(\Delta^2{\cal K}\), so \(\lc\) is determined by the coefficient-ratio shape and \(\lr\) follows under a sign condition on \(\Delta{\cal K}\).

\begin{example}[Katz-class pairs]
\label{ex:katz-between}
The binomial, Poisson, and negative-binomial laws share the GPS-power representation with coefficients \(\binom{n}{k}\), \(1/k!\), and \((r)_k/k!\). The GPS comparison gives:
\begin{enumerate}[label=\textup{(\roman*)}]
\item \(\Bin(n,p)\lc\Poi(\lambda)\), with \(\Bin\lr\Poi\iff np\le(1-p)\lambda\) and \(\Bin\st\Poi\iff(1-p)^n\ge e^{-\lambda}\).
\item \(\Bin(n,p)\lc\NB(r,\pi)\), with \(\Bin\lr\NB\iff np\le r(1-\pi)\) and \(\Bin\st\NB\iff(1-p)^n\ge\pi^r\).
\item \(\Poi(\lambda)\lc\NB(r,p)\), with \(\Poi\lr\NB\iff\lambda\le r(1-p)\) and \(\Poi\st\NB\iff e^{-\lambda}\ge p^r\).
\end{enumerate}
The coefficient ratios are log-concave. The binomial--Poisson ratio is the falling factorial \(\prod_{j=0}^{k-1}(n-j)\), the binomial--negative-binomial ratio multiplies this by \(1/(r)_k\), and the Poisson--negative-binomial ratio is \(1/(r)_k\).
\end{example}

\paragraph{Shared exponential-family representation.}
Consider
\[
w^P_k=h_P(k)e^{\theta_P T_P(k)}
\qquad\text{and}\qquad
w^Q_k=h_Q(k)e^{\theta_Q T_Q(k)}.
\]
Then
\[
{\cal K}(k) =\log\frac{h_P(k)}{h_Q(k)} +\theta_PT_P(k)-\theta_QT_Q(k).
\]
A shared carrier \(h_P=h_Q\) cancels the carrier block, leaving \({\cal K}=\theta_PT_P-\theta_QT_Q\). A shared affine statistic \(T_P=T_Q=T\) cancels the statistic block under \(\Delta^2\), so \((h_P/h_Q)\) log-concave on \(J\) implies \(P\lc Q\).

\begin{example}[CMP in the dispersion]
\label{ex:cmp-in-nu}
View \(\CMP(\lambda,\nu)\) as an exponential family with \(\eta=-\nu\) at fixed \(\lambda\), carrier \(h(k)=\lambda^k\), and statistic \(T(k)=\log(k!)\). The shared-carrier rule with \(\nu_1\le\nu_2\) gives \({\cal K}(k)=(\nu_1-\nu_2)\log(k!)\). Since \(\log(k!)\) is increasing (so \({\cal K}\) is nonincreasing, giving the \(\lr\) direction below) and convex (so \({\cal K}\) is concave, giving the \(\lc\) direction), \(\CMP(\lambda,\nu_2)\lr\CMP(\lambda,\nu_1)\) and \(\CMP(\lambda,\nu_2)\lc\CMP(\lambda,\nu_1)\).
\end{example}

\paragraph{Exponential-family versus GPS representation.}
For \(w^P_k=h_P(k)e^{\theta_PT_P(k)}\) and \(w^Q_k=b_k\theta_Q^k\) with \(T_P\) affine,
\[
\Delta^2{\cal K}(k) =\Delta^2\log\frac{h_P(k)}{b_k},
\]
so \((h_P/b_k)\) log-concave on \(J\) implies \(P\lc Q\). When \(T_P(k)=k\), this reduces to the GPS comparison with coefficient sequence \(a_k=h_P(k)\).

\begin{example}[CMP comparisons]
\label{ex:cmp-benchmarks}
View \(\CMP(\mu,\nu)\) in exponential-family representation with \(T_P(k)=k\) and \(h_P(k)=1/(k!)^\nu\). The exponential-family vs GPS comparison gives:
\begin{enumerate}[label=\textup{(\roman*)}]
\item For \(\nu\ge 1\), \(\CMP(\mu,\nu)\lc\Poi(\lambda)\) with \(\CMP\lr\Poi\iff\mu\le\lambda\). For \(\nu\le 1\), the direction reverses: \(\Poi(\lambda)\lc\CMP(\mu,\nu)\) with \(\Poi\lr\CMP\iff\lambda\le\mu\).
\item \(\CMP(\mu,\nu)\lc\Geom(p)\) for every \(\nu>0\), with \(\CMP\lr\Geom\iff\mu\le 1-p\).
\item For \(\nu\ge 1\), \(\CMP(\lambda,\nu)\lc\NB(r,p)\): the ratio \(h_P/b_k=(k!)^{1-\nu}/(r)_k\) is log-concave (log-concave numerator times the reciprocal of a log-convex Pochhammer block).
\end{enumerate}
\end{example}

\paragraph{Pochhammer-product factors.}
When both factors are products of Pochhammer or finite-support Pochhammer blocks, the log-factor ratio is a sum of logarithms with arguments linear in \(k\), and \(\Delta{\cal K}\) is a difference of harmonic terms. A sign-of-slope test at a single endpoint then governs the order.

\begin{example}[Beta-binomial vs hypergeometric]
\label{ex:bb-hyp}
Assume \(r,s>0\) and that the hypergeometric law has support \(\{0,\dots,n\}\), for instance \(B,W\ge n\). Take \(P=\Hyp(B,W,n)\) and \(Q=\BetaBin(n,r,s)\), with factors
\[
w^P_k=\binom{B}{k}\binom{W}{n-k},
\qquad
w^Q_k=\binom{n}{k}(r)_k(s)_{n-k}.
\]
The pairwise kernel \({\cal K}(k)=\log(w^P_k/w^Q_k)\) collects the four Pochhammer blocks, and its forward difference simplifies to
\[
\Delta{\cal K}(k)
=\log\frac{(B-k)(s+n-k-1)}{(W-n+k+1)(r+k)},
\]
which is nonincreasing in \(k\). Hence \(\Delta{\cal K}(k)\ge \Delta{\cal K}(n-1)\) for every \(k=0,\dots,n-1\), and \(\Delta{\cal K}(n-1)\ge 0\) reduces to \(W(r+n-1)\le s(B-n+1)\). Under this condition \({\cal K}\) is nondecreasing on \(\{0,\dots,n\}\), so the criterion \(Q\lr P\iff{\cal K}\) nondecreasing gives
\[
\BetaBin(n,r,s)\lr\Hyp(B,W,n).
\]
\end{example}

The shared-statistic exponential-family comparison admits a dual reading. \(\Poi(\lambda)\) and \(\NB(r,p)\) admit exponential-family representations with common statistic \(T(k)=k\). For \(\Poi\), \(\theta_P=\log\lambda\) and \(h_P(k)=1/k!\). For \(\NB\), \(\theta_Q=\log(1-p)\) and \(h_Q(k)=(r)_k/k!\). The carrier ratio \(h_P/h_Q=1/(r)_k\) is the reciprocal of a log-convex Pochhammer block, hence log-concave, and the shared-statistic exponential-family comparison recovers \(\Poi(\lambda)\lc\NB(r,p)\) from Example~\ref{ex:katz-between}\,\textup{(iii)} via a different route. The pairwise kernel is invariant under reparameterisation up to the additive constant absorbed by the normalising ratio: the same \({\cal K}\) arises from the GPS and exponential-family representations.

\subsection{Compound applications}
\label{sec:compound-applications}

Proposition~\ref{prop:compound-monotone} reduces the likelihood-ratio comparison of compound laws to monotonicity of \(\mathcal G_\nu\) in $n$, provided the summand law is PF$_2$. The five classical counting laws have an affine kernel, whose monotonicity reduces to the sign of its slope, so a single application of the proposition gives their $\lr$ comparisons in one statement.

\begin{example}[Compound likelihood-ratio comparison in the parameter of \(Q_\nu\)]
\label{ex:compound-rate}
Let \(F\) be PF$_2$. For each of the five  counting laws \(Q_\nu\) listed in Table~\ref{tab:compound-slopes}, inspect the factor \(w_n(\nu)\) in \(q_\nu(n)\) and differentiate its logarithm to obtain \(\mathcal G_\nu(n)\). Each \(\mathcal G_\nu\) is affine in \(n\) with the slope sign shown in the table, so monotonicity in \(n\) holds and Proposition~\ref{prop:compound-monotone} gives the \(\lr\) direction listed alongside.
\end{example}

\begin{table}[ht]
\centering
\footnotesize
\renewcommand{\arraystretch}{1.4}
\caption{Kernels \(\mathcal G_\nu(n)=\partial_\nu\log w_n(\nu)\) and \(\lr\) direction in $\nu$ for the counting laws covered by Example~\ref{ex:compound-rate}.}\label{tab:compound-slopes}
\begin{tabular}{@{}lllll@{}}
\toprule
\textbf{Counting law} \(Q_\nu\) & \textbf{Parameter} \(\nu\) & \textbf{Kernel} \(\mathcal G_\nu(n)\) & \textbf{Slope} \(\Delta\mathcal G_\nu(n)\) & \textbf{$\lr$ direction} \\
\midrule
$\Poi(\lambda)$ & $\lambda>0$ & $n/\lambda$ & $+1/\lambda$ & $C_{\nu_1}\lr C_{\nu_2}$ for $\nu_1\le\nu_2$ \\
$\Geom(p)$ on $\N_0$ & $p\in(0,1)$ & $-n/(1-p)$ & $-1/(1-p)$ & $C_{\nu_2}\lr C_{\nu_1}$ for $\nu_1\le\nu_2$ \\
$\NB(\alpha,p)$, $\alpha$ fixed & $p\in(0,1)$ & $\alpha/p-n/(1-p)$ & $-1/(1-p)$ & $C_{\nu_2}\lr C_{\nu_1}$ for $\nu_1\le\nu_2$ \\
$\Bin(n_0,p)$, $n_0$ fixed & $p\in(0,1)$ & $n/(p(1-p))-n_0/(1-p)$ & $+1/(p(1-p))$ & $C_{\nu_1}\lr C_{\nu_2}$ for $\nu_1\le\nu_2$ \\
$\mathrm{LogSeries}(\theta)$ on $\N$ & $\theta\in(0,1)$ & $n/\theta$ & $+1/\theta$ & $C_{\nu_1}\lr C_{\nu_2}$ for $\nu_1\le\nu_2$ \\
\bottomrule
\end{tabular}
\end{table}

The compound Poisson case recovers the standard ordering of compound Poisson laws in the rate parameter under the PF$_2$ summand condition. This covers the geometric-jump model used in actuarial science, and more generally any compound Poisson law with a PF$_2$ jump distribution on \(\N\). The compound geometric and compound negative-binomial cases reverse direction because \(\nu=p\) controls the success probability rather than a rate. The compound geometric case contrasts with~\cite{XiaLv2024}: relative log-concavity is unavailable for compound geometric laws with general PF$_2$ summand laws, but the likelihood-ratio comparison delivered by Example~\ref{ex:compound-rate} holds throughout.

Example~\ref{ex:compound-rate} fixes the shape parameter of $\NB(\alpha,p)$ and varies the success probability $p$. Varying instead the shape parameter $\alpha$ gives a harmonic kernel that is nondecreasing but not affine in $n$, so Proposition~\ref{prop:compound-monotone} still applies.

\begin{example}[Compound negative-binomial in the shape parameter]
\label{ex:compound-shape}
Let \(F\) be PF$_2$. Let \(C_\alpha\) be the compound law with \(N\sim\NB(\alpha,p)\), \(p\in(0,1)\) fixed, and \(\alpha>0\) varying. The \(\alpha\)-dependent factor of \(\NB(\alpha,p)\) is \(w_n(\alpha)=(\alpha)_n/n!\), so the kernel is the harmonic kernel
\[
\mathcal G_\alpha(n)=\partial_\alpha\log w_n(\alpha)=\psi(\alpha+n)-\psi(\alpha),
\]
which is nondecreasing in \(n\) for every \(\alpha>0\). Proposition~\ref{prop:compound-monotone} gives \(C_{\alpha_1}\lr C_{\alpha_2}\) for \(\alpha_1\le\alpha_2\).
\end{example}
Consequently Proposition~\ref{prop:compound-monotone} gives \(\lr\) but not, on its own, \(\lc\) in the shape parameter of the compound negative-binomial.

\begin{example}[Poisson-binomial law]
\label{ex:poisson-binomial}
Let \(X=\xi_1+\cdots+\xi_n\) with independent \(\xi_i\sim\Ber(p_i)\) and write \(\PB(\mathbf p)\) for the law of \(X\) with success-probability vector \(\mathbf p=(p_1,\dots,p_n)\in(0,1)^n\). The summands are independent Bernoullis, so the coordinate decomposition
\[
X=\xi_1+S_{-1},\qquad S_{-1}\coloneqq \xi_2+\cdots+\xi_n,
\]
exhibits \(X\) as the convolution of \(\xi_1\sim\Ber(p_1)\) and the law of \(S_{-1}\), which is independent of \(\xi_1\). The law of \(S_{-1}\) is PF$_2$, since Bernoulli laws are PF$_2$ and the PF$_2$ property is closed under convolution. Take \(\nu\coloneqq p_1\). The compound-binomial row of Table~\ref{tab:compound-slopes} with \(n_0=1\) and \(F=\delta_1\) gives \(\Ber(\nu_1)\lr\Ber(\nu_2)\) for \(\nu_1\le\nu_2\). Convolution with a PF$_2$ law preserves the likelihood-ratio order~\cite[Theorem~1.C.11]{ShakedShanthikumar}, hence
\[
\PB(\nu_1,p_2,\dots,p_n)\lr\PB(\nu_2,p_2,\dots,p_n),\qquad \nu_1\le\nu_2.
\]
Iterating coordinatewise and using transitivity of \(\lr\),
\[
\mathbf p\le\mathbf q\text{ coordinatewise}
\Longrightarrow
\PB(\mathbf p)\lr\PB(\mathbf q),
\]
on \((0,1)^n\) and, by continuity, on \([0,1]^n\). The fully symmetric case \(\mathbf p_i=(p_i,\dots,p_i)\) recovers \(\Bin(n,p_1)\lr\Bin(n,p_2)\) for \(p_1\le p_2\), already in the compound-binomial row of Table~\ref{tab:compound-slopes}.
\end{example}


\section{Conclusion}
\label{sec:discussion}

We have shown that the four stochastic orders considered here can be read from the score, or equivalently from any kernel \(K_\nu\) on the support. The score is the unique centred kernel, and any other kernel differs from it only by a parameter-dependent constant. Thus, the score and kernel can be used interchangeably in the criteria, while the kernel form often removes normalising terms and leaves a simpler state-dependent expression.

This equivalence extends beyond comparisons within a parametric family. Along joint-parameter paths, the path kernel is obtained by the chain rule. In pairwise factor-form comparisons, the kernel is the log ratio of the factors and gives a direct comparison without first choosing an interpolating family. More generally, when a law is induced from parameter-dependent factors or atom weights, the kernel can be obtained at the factor level, before centring produces the score. In compound laws, the kernel of the counting law is the parameter derivative of the log atom weights, and the compound kernel is its posterior average given the compound sum.


\printbibliography

\end{document}